\documentclass[oneside]{amsart}
\usepackage[T2A]{fontenc}
\usepackage[cp1251]{inputenc}
\usepackage{amssymb}
\usepackage{amsxtra}
\usepackage{amsmath}
\usepackage{amsthm}
\usepackage[all]{xy}

\usepackage{hyperref}

\usepackage[russian,english]{babel}

\DeclareMathOperator{\GL}{GL}

\DeclareMathOperator{\Cent}{Cent}

\DeclareMathOperator{\Hom}{Hom}
\DeclareMathOperator{\Spec}{Spec}

\DeclareMathOperator{\Gm}{{\mathbf G}_m}

\DeclareMathOperator{\der}{der}
\DeclareMathOperator{\corad}{corad}

\newcommand{\Aff}{\mathbb {A}}

\newcommand{\ad}{{ad}}

\newcommand{\Sm}{Sm}

\DeclareMathOperator{\Sing}{Sing}
\DeclareMathOperator{\A1}{\Aff^1}
\newcommand{\Hcat}{\mathcal{H}^{\A1}}

\DeclareMathOperator{\et}{\text{\it \'et}}

\newtheorem{lem}{Lemma}[section]
\newtheorem*{thm*}{Theorem}
\newtheorem{thm}[lem]{Theorem}

\newtheorem{cor}[lem]{Corollary}
\theoremstyle{definition}{  \newtheorem{rem}[lem]{Remark}  }
\theoremstyle{definition}{  \newtheorem{dfn}[lem]{Definition}  }

\newcommand{\st}{\scriptstyle}



\begin{document}

\title{$\Aff^1$-invariance of non-stable $K_1$-functors in the equicharacteristic case}
\author{Anastasia Stavrova}
\thanks{The author is a winner of the contest ``Young Russian Mathematics''.
The work was supported by RFBR~19-01-00513.}
\address{Chebyshev Laboratory, Department of Mathematics and Computer Science, St. Petersburg State University,
14th Line V.O. 29B, 199178 Saint Petersburg, Russia}
\email{anastasia.stavrova@gmail.com}
\subjclass[2010]{19B99, 20G35,  14L15, 20G99}
\keywords{isotropic reductive group, non-stable $K_1$-functor, Whitehead group, Serre-Grothendieck conjecture}

\selectlanguage{english}
\maketitle

\begin{abstract}
We apply the techniques developed by I. Panin for the proof of the equicharacteristic case of the
Serre--Grothendieck conjecture for isotropic reductive groups
(I. Panin, A. Stavrova, N. Vavilov, 2015; I. Panin, 2019) to obtain
similar injectivity and $\Aff^1$-invariance theorems for
non-stable $K_1$-functors associated to isotropic reductive groups.
Namely, let $G$ be a reductive group over a commutative ring $R$. We say that $G$ has isotropic rank $\ge n$,
if every non-trivial normal semisimple $R$-subgroup of $G$ contains $(\Gm_{,R})^n$. We show that if $G$ has isotropic rank $\ge 2$
and $R$ is a regular domain containing a field, then $K_1^G(R[x])=K_1^G(R)$,
where $K_1^G(R)=G(R)/E(R)$ is the corresponding
non-stable $K_1$-functor, also called the Whitehead group of $G$. If $R$ is, moreover, local, then we show that
$K_1^G(R)\to K_1^G(K)$ is injective, where $K$ is the field of fractions of $R$.
\end{abstract}

\section{Introduction}

Let $R$ be a commutative ring with 1.
Let $G$ be a reductive group scheme over $R$ in the sense of~\cite{SGA3}. We say that $G$ has isotropic rank $\ge n$,
if every non-trivial normal semisimple $R$-subgroup of $G$ contains $(\Gm_{,R})^n$.

If $G$ is not a torus, the assumption that $G$ has isotropic rank $\ge 1$ implies that $G$ contains a
proper parabolic subgroup~\cite[Exp. XXVI, Proposition 6.1]{SGA3}.
For any reductive group $G$ over $R$ and a parabolic subgroup $P$ of $G$,
one defines the elementary subgroup $E_P(R)$ of $G(R)$ as the subgroup
generated by the $R$-points of the unipotent radicals of $P$ and of an opposite parabolic subgroup
$P^-$, and considers the
corresponding non-stable $K_1$-functor $K_1^{G,P}(R)=G(R)/E_P(R)$~\cite{PS,St-poly}. It does not depend on the
choice of $P^-$ by~\cite[Exp. XXVI Cor. 1.8]{SGA3}.
In particular, if $A=k$ is a field and $P$ is minimal, $E(k)$ is nothing but the group $G(k)^+$ introduced
by J. Tits~\cite{Tits64}, and $K_1^G(k)$ is the subject of the Kneser--Tits problem~\cite{Gil}. If $G=\GL_n$
and $P$ is a Borel subgroup,
then $K_1^G(R)=\GL_n(R)/E_n(R)$, $n\ge 1$, are the usual non-stable $K_1$-functors of algebraic $K$-theory.
If $G$ has isotropic
rank $\ge 2$, then $K_1^{G,P}(R)$ is independent of $P$ by the main result of~\cite{PS},
and we denote it by $K_1^G(R)$. If $G$ is a torus, we define $K_1^G(R)=G(R)$ for coherence.
See~\S~\ref{sec:prel} for a formal definition and further properties of $K_1^G$.

In~\cite{St-poly} we proved that
if $G$ is a reductive
group over a field $k$ having isotropic rank $\ge 2$, then
$$K_1^G(k)=K_1^G(k[X_1,\ldots,X_n])\qquad\mbox{for any}\quad n\ge 1.
$$
This implied the following two statements. First, provided that $k$ is perfect, one has $K_1^G(A)=K_1^G(A[x])$ for any
regular ring $A$ containing $k$. Second, provided that $k$ is infinite and perfect, one has
$\ker(K_1^G(A)\to K_1^G(K))=1$ for any regular local ring $A$ containing $k$, where $K$ is the field of fractions of $A$.
Those results generalized several earlier results on split, i.e. Chevalley--Demazure,
reductive groups; see~\cite{St-poly} for a historical survey.

In the present text we show how the techniques developed by I. Panin for the proof of the equicharacteristic case of the
Serre--Grothendieck conjecture~\cite{PaStV,Pa-ICM,Pa-GSfin1,Pa-izvGS}
allow to extend these results to the case where $G$ is defined over a regular ring $A$ containing a field $k$,
but not necessarily over $k$ itself, and $k$ is an arbitrary field. The main results
are the following.

\begin{thm}\label{thm:A1-inv}
Let $A$ be a regular ring containing a field $k$.
Let $G$ be a reductive group scheme over $A$ of isotropic rank $\ge 2$.
Then there is a natural isomorphism
$K_1^G(A)\cong K_1^G(A[x])$.

\end{thm}

\begin{thm}\label{thm:inj}
Let $A$ be a semilocal regular domain containing a field $k$, and let $K$ be the field of fractions of $A$.
 Let $G$ be a reductive group scheme over $A$ of isotropic rank $\ge 2$.
Then the natural homomorphism $K_1^G(A)\to K_1^G(K)$ is injective.
\end{thm}

As a corollary of these two theorems, we also obtain some results on the values of non-stable $K_1$-functors on
Laurent polynomials and power series rings; see Corollary~\ref{cor:Lau} in~\S~\ref{sec:main}
and~\S~\ref{sec:ser}. In the end of the paper we apply the general results of A. Asok,
M. Hoyois and M. Wendt~\cite{AHW} to give an interpretation of
$K_1^G$ in terms of the $A^1$-homotopy theory of F. Morel and V. Voevodsky~\cite{MoV}; see Theorem~\ref{thm:repr}.

\section{Properties of $K_1^G$ over general commutative rings}\label{sec:prel}

Let $R$ be a commutative ring with 1. Let $G$ be a reductive group scheme over $R$, and
let $\pi:G\to G^{\ad}=G/\Cent(G)$ be the canonical homomorphism of $G$
onto its adjoint group. The correspondence $H\mapsto \pi(H)$
between semisimple normal subgroups of $G$ and $G^{\ad}$ is bijective, with the
inverse correspondence $H\mapsto\der(\pi^{-1}(T\cdot H))$, where $T$ is any fixed maximal torus of $G^{\ad}$
and $\der$ denotes the derived subgroup of a reductive group
in the sense of~\cite[Exp. XXII, \S 6]{SGA3}. The bijectivity follows by faithfully flat descent
from~\cite[Exp. XXII, Lemme 5.1.5, Lemme 5.2.7, Corollaire 5.3.5]{SGA3}. In particular, $G$ has isotropic rank $\ge n$
if and only if $G^{\ad}$ has isotropic rank $\ge n$.

For adjoint groups, semisimple normal subgroups can be described more explicitly.
By~\cite[Exp. XXIV, Proposition 5.10]{SGA3} $G^{\ad}$ is isomorphic to a direct product of Weil restrictions
$G_i=\prod_{R'_i/R}H_i$, where $R'_i/R$ is finite \'etale and $H_i$ is an adjoint simple group over $R'_i$.
We can assume without loss of generality that every $G_i$ cannot be decomposed further into a product of such factors.
Then every non-trivial semisimple normal subgroup of $G^{\ad}$ is a direct product of several factors of this
decomposition.
By adjunction of the Weil restriction and base change, the fact that $G_i$ contains $(\Gm_{,R})^n$ implies
that $H_i$ contains $(\Gm_{,R'_i})^n$. The converse is also true, since
$(\Gm_{,R})^n$ is a subgroup of $\prod_{R'_i/R}(\Gm_{,R'_i})^n$.
Therefore, for any ring homomorphism $R\to S$,
if $G$ has isotropic rank $\ge n$, $G_S$ also does, since it is a product of Weil restrictions of groups
$(H_i)_S$ (cf.~\cite[\S 7.6]{Neron-book}).

Let $P$ be a parabolic subgroup of $G$ in the sense of~\cite{SGA3}.
Since the base $\Spec R$ is affine, the group $P$ has a Levi subgroup $L_P$~\cite[Exp.~XXVI Cor.~2.3]{SGA3}.
There is a unique parabolic subgroup $P^-$ in $G$ which is opposite to $P$ with respect to $L_P$,
that is $P^-\cap P=L_P$, cf.~\cite[Exp. XXVI Th. 4.3.2]{SGA3}.  We denote by $U_P$ and $U_{P^-}$ the unipotent
radicals of $P$ and $P^-$ respectively.

\begin{dfn}\label{defn:E_P}\cite{PS}
The \emph{elementary subgroup $E_P(R)$ corresponding to $P$} is the subgroup of $G(R)$
generated as an abstract group by $U_P(R)$ and $U_{P^-}(R)$. We denote by $K_1^{G,P}(R)=G(R)/E_P(R)$ the pointed set of
cosets $gE_P(R)$, $g\in G(R)$.
\end{dfn}

Note that if $L'_P$ is another Levi subgroup of $P$,
then $L'_P$ and $L_P$ are conjugate by an element $u\in U_P(R)$~\cite[Exp. XXVI Cor. 1.8]{SGA3}, hence
the group
$E_P(R)$ and the set $K_1^{G,P}(R)$ do not depend on the choice of a Levi subgroup or an opposite subgroup
$P^-$ (and so we do not include $P^-$ in the notation).

\begin{dfn}
A parabolic subgroup $P$ in $G$ is called
\emph{strictly proper}, if it intersects every non-trivial normal semisimple subgroup of $G$ properly.
\end{dfn}

Similarly to the case of semisimple normal subgroups,
the correspondence $P\mapsto \pi(P)$ between parabolic subgroups of $G$ and $G^{\ad}$
is bijective (cf.~\cite[Exp. XXVI, \S 3]{SGA3}). Consequently, $G$ has a strictly proper parabolic subgroup if
and only if $G^{\ad}$ does.

If $G$ has isotropic rank $\ge 1$, and $G$ is not a torus, then $G$ contains
a strictly proper parabolic $R$-subgroup $P$. Indeed,
by the above remarks, we can assume that $G=G^{\ad}$ is an adjoint
reductive group isomorphic to a direct product of Weil restrictions
$G_i=\prod_{R'_i/R}H_i$, where $R'_i/R$ is finite \'etale and $H_i$ is an adjoint simple group over $R'_i$. Let $S\le G$ be the 1-dimensional split subtorus embedded diagonally into the product of 1-dimensional split subtori
in $G_i$ that exist by assumption. By~\cite[Exp. XXVI, Proposition 6.1]{SGA3} there is a
parabolic subgroup $P$ of $G$ such that the centralizer $L=\Cent_G(S)$ of $S$ in $G$ is a Levi subgroup of $P$.
Since $S$ acts faithfully on every $G_i$, the subgroup $L$, and hence $P$,
intersects properly every factor $G_i$, and hence every semisimple normal subgroup of $G$.

If $G$ has isotropic rank $\ge 2$, then for any strictly parabolic subgroup $P$, the functor $K_1^{G,P}$ is group-valued
and independent of $P$, as evidenced by the the following result.

\begin{thm}\label{th:PS-normality}\cite[Lemma 12, Theorem 1]{PS}
Let $G$ be a reductive group over a commutative ring $R$, and let $A$ be a commutative $R$-algebra.
If for any maximal ideal $m$ of $R$ the isotropic rank of $G_{R_m}$ is $\ge 2$,
then the subgroup $E_P(A)$ of $G(A)$ is the same for any
strictly proper parabolic $A$-subgroup $P$ of $G_A$, and is normal in $G(A)$.
\end{thm}

\begin{dfn}
Let $G$ be a reductive group of isotropic rank $\ge 2$ over a commutative ring $R$.
If $G$ is not a torus, then
for any $R$-algebra $A$,
we call the subgroup $E(A)=E_P(A)$, where $P$ is a strictly proper parabolic subgroup of $G$ over $R$,
the \emph{elementary subgroup}  of $G(A)$. If $G$ is a torus, we set $E(A)=1$.
The functor $K_1^G$ on the category of commutative $R$-algebras $A$, given by $K_1^G(A)=G(A)/E(A)$,
is called the \emph{non-stable $K_1$-functor} associated to $G$.
\end{dfn}

We will use the following two properties of $K_1^G$ established in~\cite{St-poly,St-serr}.
The following lemma was established in~\cite[Corollary 5.7]{Sus} for $G=\GL_n$ (the proof goes through for any
torus $G$ without any changes).

\begin{lem}\label{lem:inj-f}\cite[Lemma 2.7]{St-serr}
Let $A$ be a commutative ring, and let $G$ be a reductive group scheme over $A$ of isotropic rank $\ge 1$.
Assume moreover that for every maximal ideal $\mathfrak{m}\subseteq A$, the reductive $A_m$-group
$G_{A_\mathfrak{m}}$ has isotropic rank $\ge 2$.
Then for any monic polynomial $f\in A[x]$ the natural homomorphism
$$
K_1^G(A[x])\to K_1^G(A[x]_f)
$$
is injective.
\end{lem}

The following statement was proved for $G=\GL_n$, $n\ge 3$, by A. Suslin~\cite[Th. 3.1]{Sus}.
For the case of split semisimple groups the same result was obtained by E. Abe~\cite[Th. 1.15]{Abe}.

\begin{lem}\label{lem:PS-17}\cite[Lemma 17]{PS}
Let $A$ be a commutative ring, and let $G$ be a reductive group scheme over $A$ of
isotropic rank $\ge 1$.
Assume moreover that for every maximal ideal $\mathfrak{m}\subseteq A$, the reductive $A_m$-group
$G_{A_\mathfrak{m}}$ has isotropic rank $\ge 2$.
Then for any $g(X)\in G(A[X])$ such that $g(0)\in E(A)$ and
$F_m(g(X))\in E(A_m[X])$ for all maximal ideals $m$ of $A$, one has $g(X)\in E(A[X])$.
\end{lem}
\begin{proof}
If $G$ is not a torus, the claim is proved as in the proof of~\cite[Lemma 17]{PS}. If $G$ is a torus, then the claim follows from the
fact that $G$ is a sheaf for Zariski topology.
\end{proof}

The following lemma is a straightforward extension
of~\cite[Lemma 2.4]{Vo} for $G=\GL_n$ and~\cite[Lemma 3.7]{Abe} for split reductive groups.

\begin{lem}\label{lem:Nis-square}
Let $G$ be a reductive group of isotropic rank $\ge 2$ over a Noetherian commutative ring $B$.
Let $\phi:B\to A$ be a homomorphism of commutative rings, and $h\in B$ be
such that $\phi:B/hB\to A/\phi(h)A$ is an isomorphism and the restriction of $h$ to every connected component of $B$
is non-nilpotent.
Assume moreover that the commutative square
\begin{equation*}
\xymatrix@R=15pt@C=20pt{
\Spec A_{\phi(h)}\ar[r]^{F_{\phi(h)}}\ar[d]^{\phi}&\Spec A\ar[d]^{\phi}\\
\Spec B_h\ar[r]^{F_h}&\Spec B\\
}
\end{equation*}
is a distinguished Nisnevich square in the sense of~\cite[Def. 3.1.3]{MoV}.
Then the sequence of pointed sets
$$
K_1^G(B)\xrightarrow{\st (F_h,\phi)} K_1^G(B_h)\times K_1^G(A)\xrightarrow{\st (g_1,g_2)\mapsto \phi(g_1)
F_f(g_2)^{-1}}K_1^G(A_{\phi(h)})
$$
is exact.
\end{lem}
\begin{proof}
If $G$ is a torus, then $K_1^G(-)=G(-)$ by definition, and the claim follows immediately from the fact that $G$ is a sheaf for the fpqc
topology~\cite[Theorem 2.55]{FGA-explained}.
If $G$ is not a torus, we use~\cite[Corollary 3.5]{St-poly}.
In order to use this result,
we only need to check that $G$ has a strictly proper parabolic subgroup $P$ such that
the system
of relative roots $\Phi_P$ in the sense of~\cite{St-poly,St-serr} is has rank $\ge 2$ everywhere on $B$.
Since $B$ is Noetherian, $h$ is non-nilpotent on every connected component of $B$, and $K_1^G$ commutes with
finite direct products of rings, we can check
the claim of the lemma individually for every connected component of $B$. In other words, we can assume that
$B$ is connected. Let $G^{\ad}=G/\Cent(G)$ be the adjoint group corresponding to $G$.
The construction of $\Phi_P$ in~\cite[Lemma 3.6 and Definition 3.7]{St-serr} implies that $\Phi_P$ the same for
$P$ and $P/\Cent(G)$. Consequently, we can assume that $G=G^{\ad}$ is an adjoint
reductive group. Then $G$ is isomorphic to a direct product of Weil restrictions
$G_i=\prod_{B'_i/B}H_i$, where $B'_i/B$ is finite \'etale and $H_i$ is an adjoint simple group over $B'_i$, and
the factors $G_i$ cannot be decomposed further.
Let $S\le G$ be the 2-dimensional split subtorus embedded diagonally into the product of 2-dimensional split subtori
in $G_i$ that exist by assumption. By~\cite[Exp. XXVI, Proposition 6.1]{SGA3} there is a
parabolic $B$-subgroup $P$ of $G$ such that the centralizer $L=\Cent_G(S)$ of $S$ in $G$ is a Levi subgroup of $P$.
Since $S$ is contained in the center of $L$, and acts faithfully on every $G_i$,
by~\cite[Lemma 3.6]{St-serr} the system of relative roots $\Phi_P$ has rank $\ge 2$ over $B$.

\end{proof}

\section{Proof of the main results}\label{sec:main}

The following theorem proved in~\cite{Pa-GSfin1} extends~\cite[Theorem 7.1]{PaStV} to arbitrary reductive groups $G$
and arbitrary fields $k$.

\begin{thm}\label{thm:panin}\cite[Theorem 2.5]{Pa-GSfin1}\\
(i) Let $X$ be an affine $k$-smooth irreducible $k$-variety,
let $U=Spec(\mathcal O_{X,\{x_1,x_2,\dots,x_n\}})$ be the spectrum of the semilocal ring of $X$ at a finite
set of closed points $x_1,x_2,\ldots,x_n$, $n\ge 1$,
and let $0\neq f\in \mathcal O_{X}(X)$ be such that $x_1,\ldots,x_n\not\in X_f$.
Then, possibly after replacing $X$ by a smaller affine open neighborhood of $U$, one
can find a $U$-scheme $Y$ which is finite \'etale neighborhood
\begin{equation*}
\label{eq:neighb}
    \xymatrix{
     (\Aff^1_U)  && Y  \ar[ll]_{\tau}                                      &\\
       && U\ar[u]^\delta \ar[ull]_{0}  &\\
}
\end{equation*}
of the $U$-point $0\in \Aff^1_U$,
and a morphism $p:Y\to X$, such that $Y\xrightarrow{\tau} \Aff^1_U$
together with a principal open $(\Aff^1_U)_h\to \Aff^1_U$, where
$h\in O_{X,\{x_1,x_2,\dots,x_n\}}[t]$ is a monic polynomial with $h(1)\in(O_{X,\{x_1,x_2,\dots,x_n\}})^\times$,
form an elementary distinguished Nisnevich square of $U$-schemes in the sense of
\cite[Def. 3.1.3]{MoV},
and the following diagram commutes.
\begin{equation}
\label{eq:nis-square}
    \xymatrix{
       (\Aff^1_U)_{h}  \ar[d]_{inc} && Y_{\tau^*(h)} \ar[ll]_{\tau|_{Y_{\tau^*(h)}}}  \ar[d]_{inc} \ar[rr]^{p|_{Y_{\tau^*(h)}}} && X_f  \ar[d]_{inc}   &\\
     (\Aff^1_U)  && Y  \ar[ll]_{\tau} \ar[rr]^{p} && X                                     &\\
       && U\ar[u]^\delta \ar[ull]_{0} \ar@{=}[rr]&& U\ar[u]^{inc} &\\
    }
\end{equation}

(ii) Let $G$ be a reductive group scheme over $X$, and denote by $G_U$ its restriction to $U$. Then one can
choose $Y$ in (i) in such a way that, moreover, the base change $p^*(G)$ of $G$ to $Y$ is $Y$-isomorphic
to the restriction $\tau^*(H)$ to $Y$ of the ``constant'' group scheme $H=G_U\times_U\Aff^1_U$ over $\Aff^1_U$.

\end{thm}

\begin{thm}\label{thm:main}
Let $k$ be a field,
let $A$ be a semilocal ring of several closed points on a smooth irreducible $k$-variety, and let $K$ be the field of fractions of $A$.
Let $G$ be a reductive group over $A$ such that every semisimple normal subgroup of $G$ contains $(\Gm)^2$.
 Then for any commutative $k$-algebra $B$, the natural map
$$
K_1^G(A\otimes_k B)\to K_1^G(K\otimes_k B)
$$
has trivial kernel.
\end{thm}
\begin{proof}
Let $g\in\ker (K_1^G(A\otimes_k B)\to K_1^G(K\otimes_k B))$. By~\cite[Tag 01ZU]{Stacks}
there are a smooth irreducible affine $k$-variety
 $X=\Spec(C)$ and $f\in C$
such that
$A$ is a semilocal ring of several closed points on $X$. Since $G$ is finitely presented, and
$K_1^G$ commutes with filtered direct limits of rings,
we can assume that $B$ is a finitely generated $k$-algebra, $G$ is defined and
has isotropic rank $\ge 2$ over $C$, the element $g\in K_1^G(A\otimes_k B)$
is the image of an element $g_0\in\ker(K_1^G(C\otimes_k B)\to K_1^G(C_f\otimes_k B))$.
Apply Theorem~\ref{thm:panin} (i). Replace all schemes in the diagram~\eqref{eq:nis-square} by their fiber products
with $\Spec B$ over $\Spec k$.
This gives an element
$$
p^*(g_0)\in \ker\bigl(K_1^G(O_Y(Y)\otimes_k B)\to K_1^G(O_Y(Y)_{\tau^*(h)}\otimes_k B)\bigr),
$$
where we write $G$ instead of $p^*(G)$ thanks to
Theorem~\ref{thm:panin} (ii).
Note that the polynomial $h$ is non-nilpotent on every connected component of the Noetherian ring
$A[x]\otimes_k B=(A\otimes_k B)[x]$.
 By Lemma~\ref{lem:Nis-square} there is an element
$$
\tilde g\in\ker\left(K_1^G(A[x]\otimes_k B)\to K_1^G(A[x]_h\otimes_k B)\right),
$$
such that $\tau^*(\tilde g)=p^*(g_0)$.
 Since $h$ is a monic polynomial,
by Lemma~\ref{lem:inj-f} we have $\tilde g=1$.
By the commutativity of the diagram~\eqref{eq:nis-square} we have
$\tilde g|_{x=0}=\delta^*(p^*(g_0))=g$. Hence $g=1$.
\end{proof}

\begin{proof}[Proof of Theorem~\ref{thm:A1-inv}]
Clearly, we can assume that $k$ is a finite field or $\mathbb{Q}$
without loss of generality. The embedding $k\to A$ is geometrically regular, since $k$ is perfect~\cite[(28.M), (28.N)]{Mats}.
Then by Popescu's theorem~\cite{Po90,Swan} $A$ is a filtered direct limit of smooth
$k$-algebras $R$.
Since the group scheme $G$ and the unipotent radicals of its parabolic subgroups are finitely presented over $A$,
the functors $G(-)$ and $E(-)=E_P(-)$ commute with filtered direct limits. Hence we can replace $A$ by
a smooth $k$-algebra $R$. By the local-global principle Lemma~\ref{lem:PS-17}, to show that $K_1^G(R)=K_1^G(R[x])$,
it is enough to show that for every maximal localization $R_m$ of $R$. Let $K$ be the field of fractions of $R_m$.
By Theorem~\ref{thm:main} the maps $K_1^G(R_m)\to K_1^G(K)$ and $K_1^G(R_m[x])\to K_1^G(K[x])$ are injective.
By~\cite[Theorem 1.2]{St-poly} $K_1^G(K)=K_1^G(K[x])$. Hence $K_1^G(R_m)=K_1^G(R_m[x])$, and we are done.
\end{proof}

\begin{proof}[Proof of Theorem~\ref{thm:inj}]
As in the proof of Theorem~\ref{thm:A1-inv} we are reduced to the case where $A$ is a filtered direct limit of smooth
$k$-algebras $R$, and $G$ is defined and has isotropic rank $\ge 2$ over $R$. Also, since $A$ is a semilocal domain,
we can assume that $R$ is a domain with a fraction field $L$,
and an element $g\in\ker(K_1^G(A)\to K_1^G(K))$ comes from an element $g'\in K_1^G(R)$
that vanishes in $K_1^G(L)$. By Theorem~\ref{thm:main}, $g'$ vanishes in every semilocalization of $R$ at a finite
set of maximal ideals, and hence in every semilocalization of $R$ at a finite
set of prime ideals. Since the map $K_1^G(R)\to K_1^G(A)$ factors through such a semilocalization of $R$, it
follows that $g=1$.
\end{proof}

\begin{cor}\label{cor:Lau}
Let $A$ be a semilocal regular domain containing a field $k$.
 Let $G$ be a simply connected semisimple group scheme over $A$ of isotropic rank $\ge 2$.
Then
$K_1^G(A)=K_1^G(A[x_1^{\pm 1},\ldots,x_n^{\pm 1}])$.
\end{cor}
\begin{proof}
Recall that for any field $K$ and any simply connected semisimple $G$ of isotropic rank $\ge 2$ over $K$ one has
$K_1^G(K[x_1^{\pm 1},\ldots,x_n^{\pm 1}])\cong K_1^G(K)$ by~\cite{St-poly}.
Let $K$ be the field of fractions of $A$.
Then the claim follows from
Theorem~\ref{thm:main} exactly as Theorem~\ref{thm:inj}, via the following diagram:
\begin{equation*}
\xymatrix{
K_1^G(A[x_1^{\pm 1},\ldots,x_n^{\pm 1}])\ar[rrr]^{\hspace{28pt}  x_1=\ldots=x_n=1}\ar[d] &&& K_1^G(A)\ar[d]^{}  &\\
K_1^G(K[x_1^{\pm 1},\ldots,x_n^{\pm 1}])\ar[rrr]^{\hspace{28pt} x_1=\ldots=x_n=1}&&& K_1^G(K).
}
\end{equation*}
\end{proof}

\begin{rem}
It is clear that Theorem~\ref{thm:panin} can be applied to deduce analogs of Theorems~\ref{thm:A1-inv} and~\ref{thm:inj}
for any reasonably good functor defined in terms of a reductive group scheme and satisfying properties similar to
Lemmas~\ref{lem:inj-f},~\ref{lem:PS-17},~\ref{lem:Nis-square}. The paradigmal example is the functor $H^1_{\et}(-,G)$,
in which case Theorem~\ref{thm:inj} corresponds to the Serre--Grothendieck conjecture
(the non-isotropic cases involve additional modification of the counterpart of Lemma~\ref{lem:inj-f}).
One can axiomatize this approach similarly to the ``constant'' case~\cite[Th\'eor\`eme 1.1]{CTO}, however, the axioms
are, naturally, more cumbersome.
\end{rem}

\section{Power series rings}\label{sec:ser}

We will need the following analog of a well-known theorem of Quillen on projective modules~\cite[Theorem 3]{Q}.

\begin{thm}\cite[Theorem 1.3]{PaStV}\label{thm:H1-f}
Let $B$ be a commutative ring, and let $G$ be a simply connected semisimple group over $B$ of isotropic rank
$\ge 1$. Let $f\in B[x]$ be a monic polynomial. Then the natural map of \'etale
cohomology sets
$H^1_{\et}(B[x],G)\to H^1_{\et}(B[x]_f,G)$
has trivial kernel.
\end{thm}
\begin{proof}
Assume first that $B$ is a local ring.
Since $G$ is simply connected, similarly to the adjoint case,
by~\cite[Exp. XXIV, Proposition 5.10]{SGA3} implies that $G$ is isomorphic to a direct product of Weil restrictions
$G_i=\prod_{B'_i/B}H_i$, where $R'_i/R$ is finite \'etale and $H_i$ is a simply connected simple group over $B'_i$,
i.e. the Dynkin diagram of $H_i$ over every geometric point
of $\Spec(B'_i)$ is irreducible. Again as in the adjoint case (see~\S 2), the assumption that $G$ has sotropic rank $\ge 1$
implies that $H_i$ contains $\Gm_{,B'_i}$. Then
$H^1_{\et}(B'_i[x],G)\to H^1_{\et}(B'_i[x]_f,H_i)$ has trivial kernel by~\cite[Proposition 5.2.2 (i)]{Ces-GS}. Then
the same claim for $G$ follows from a version of Faddeev--Shapiro's lemma~\cite[Exp. XXIV, Proposition 8.4]{SGA3}.

Now assume that $B$ is arbitrary. Since $G$ is finitely presented, and $H^1_{\et}(-,G)$ commutes with
filtered direct limits~\cite{Margaux-Chech}, we can assume that $B$ is Noetherian.
Since $G$ is semisimple, it is $B$-linear by~\cite[Corollary 3.2]{Thomason}. Then Lemma~\ref{lem:loc-f} below
finishes the proof.
\end{proof}

\begin{lem}\label{lem:loc-f}
Let $B$ be a Noetherian ring, and let $G$ be a $B$-linear flat group scheme.
Let $f\in B[x]$ be a monic polynomial.
Assume that for every maximal ideal $m$ of $B$
the natural maps of \'etale
cohomology sets $H^1_{\et}(B_m[x],G)\to H^1_{\et}(B_m[x]_{f},G)$
and $H^1_{\et}(B_m[x],G)\to H^1_{\et}(B_m[x]_x,G)$ have trivial kernels.
Then the map $H^1_{\et}(B[x],G)\to H^1_{\et}(B[x]_f,G)$
has trivial kernel.
\end{lem}
\begin{proof}
Let $\xi\in H^1_{\et}(B[x],G)$ be in the kernel.
Set $y=x^{-1}$ and choose $g(y)\in B[y]$ so that $x^{\deg(f)}g(y)=f(x)$. Then $g(0)=1\in B^\times$, and
$B[x]_{xf}=B[y]_{yg}$.
Hence $\Spec(B[y])=\Spec(B[y]_y)\cup\Spec(B[y]_g)$.
Extend $\xi|_{\Spec(B[x]_x)}=\xi|_{\Spec(B[y]_y)}$ to a bundle $\eta$ on $\Spec(B[y])$
by gluing it to the trivial bundle on $\Spec(B[y]_g)$ (this can be done e.g. by~\cite[Proposition 2.6 (iv)]{CTO}).
By assumption, for any maximal ideal $m$ of $B$ the $G$-bundle $\xi|_{B_m[x]}$ is trivial, hence
$\eta|_{\Spec(B_m[y]_y)}$ is trivial. Then, again by assumption, $\eta|_{B_m[y]}$ is trivial.
Since $G$ is $B$-linear, by~\cite[Theorem 3.2.5]{AHW}
the fact that for any maximal ideal $m$ of $B$ the $G$-bundle $\eta|_{B_m[y]}$ is trivial implies
that  $\eta$ is extended from $B$. However, $\eta$ is trivial at $y=0$ by construction,
so $\eta$ is trivial. Hence $\xi$ is trivial at $x=y=1$. Since $\xi|_{B_m[x]}$ is also trivial for any $m$,
by~\cite[Theorem 3.2.5]{AHW} $\xi$ is also extended from $B$. Hence $\xi$ is trivial.
\end{proof}

\begin{rem}
Theorem~\ref{thm:H1-f} was first established in~\cite[Theorem 1.3]{PaStV}
with the following additional assumptions. First, $B$ was assumed to be a
Noetherian $k$-algebra, where $k$ was an arbitrary field. Second, $G$ was assumed to be simple (that is, the root system of
$G$ over any geometric point of $\Spec(B)$ is irreducible). Third, it was assumed that
$f(1)\in B^\times$. The assumptions that $B$ is Noetherian and $G$ is simple are unsubstantial, as explained
in our proof of Theorem~\ref{thm:H1-f}.
The fact that $B$ is a $k$-algebra was used in the proof of~\cite[Theorem 1.3]{PaStV} at one point only. Namely,
in~\cite[Lemma A.1]{PaStV}, we referred to the classical results of W. J. Haboush~\cite{Hab} and M. Nagata~\cite{Na} that a quotient of
a split reductive group over a field $k$ by a reductive subgroup is representable by an affine scheme. It was observed
in~\cite{Ces-GS} that this result
was generalized to arbitrary reductive group schemes over Noetherian rings by C. S. Seshadri~\cite{Ses-geomred} and
J. Alper~\cite[9.4.1 and 9.7.5]{Alp}.
The condition $f(1)\in B^\times$ was used in the proof of~\cite[Theorem 1.3]{PaStV} in order to deduce the case
of a not necessarily local ring $B$ from the local ring case, the latter being established without this assumption.
This argument is replaced by the new Lemma~\ref{lem:loc-f}.
\end{rem}

\begin{cor}\label{cor:Lau-square}
Let $B$ be a commutative ring, and let $G$ be a simply connected semisimple group over $B$ of isotropic rank
$\ge 1$. Then $G\bigl(B((x))\bigr)=G(B[[x]])G(B[x^{\pm 1}])$. If, moreover, $G$ has isotropic rank $\ge 2$,
then the sequence of pointed sets
$$
1\longrightarrow K_1^G(B[x])\xrightarrow{\st g\mapsto (g,g)} K_1^G(B[[x]])\times K_1^G(B[x,x^{-1}])
\xrightarrow{\st (g_1,g_2)\mapsto g_1{g_2}^{-1}} K_1^G\bigl(B((x))\bigr)\to 1
$$
is exact.
\end{cor}
\begin{proof}
Take an element $g\in G\bigl(B((x))\bigr)$.
The schemes $\Spec(B[x])$, $\Spec\bigl(B((x))\bigr)$, $\Spec\bigl(B[[x]]\bigr)$ and $\Spec(B[x^{\pm 1}])$
together form a patching square for $G$-torsors~\cite[Lemma 2.2.11]{BC}.
Then one can patch $G_{B[[x]]}$ and $G_{B[x^{\pm 1}]}$ by means of a shift by $g$ to obtain a $G$-torsor over $B[x]$.
Since $G$ is smooth, this torsor is \'etale-locally trivial~\cite[11.7]{Gr-BrauerIII}.
By Theorem~\ref{thm:H1-f} this torsor is trivial, since its restriction to $B[x^{\pm 1}]$ is trivial.
Therefore, $g\in G(B[[x]])G(B[x^{\pm 1}])$.

The second claim follows from the first and~\cite[Corollary 3.2]{St-serr}.
\end{proof}

The following result, in particular, generalizes a theorem of Ph. Gille~\cite[Th\'eor\`eme 5.8]{Gil} that
$K_1^G(k)=K_1^G\bigl(k((x))\bigr)$ for any isotropic simply connected semisimple group $G$ over a field $k$.

\begin{thm}\label{thm:semiloc-reg}
Let $B$ be a semilocal regular ring containing a field, and let $G$ be a simply
connected semisimple group over $B$ of
isotropic rank $\ge 2$. Then
$$
K_1^G(B)=K_1^G\bigl(B((x))\bigr)=K_1^G(B[[x]])=K_1^G(B[x])=K_1^G(B[x^{\pm 1}]).
$$
\end{thm}
\begin{proof}
Since $K_1^G$ commutes with finite products, we can assume that $B$ is a domain.
It follows from Theorem~\ref{thm:A1-inv} and Corollary~\ref{cor:Lau} that
$K_1^G(B)=K_1^G(B[x])=K_1^G(B[x^{\pm 1}])$.

Let $K$ be the field of fractions of $B$. Since $B[[x]]$ and $K[[x]]$ are semilocal regular rings containing a field,
by Theorem~\ref{thm:main} the maps $K_1^G(B[[x]])\to K_1^G\bigl(K((x))\bigr)$ and
$K_1^G(K[[x]])\to K_1^G\bigl(K((x))\bigr)$ are injective. Consequently, also the map
$K_1^G(B[[x]])\to K_1^G\bigl(K[[x]]\bigr)$ is injective.
By~\cite[Lemme 4.5, p. 983--15]{Gil} one has
$$
\ker\bigl(G(K[[x]])\xrightarrow{x=0}G(K)\bigr)\subseteq E\bigl(K((x))\bigr),
$$
and hence $K_1^G(K[[x]])=K_1^G(K)$. It follows that $K_1^G(B[[x]])\xrightarrow{x=0}K_1^G(B)$ is injective,
and hence $K_1^G(B[[x]])=K_1^G(B)$. Finally, since
$K_1^G(B)\to K_1^G(B[[x]])\to K_1^G(K((x)))$ is injective, it follows that $K_1^G(B)\to K_1^G\bigl(B((x))\bigr)$ is
injective.

By Corollary~\ref{cor:Lau-square} we have $G(B((x)))=G(B[x^{\pm 1}])G(B[[x]])$, therefore, $K_1^G(B[[x]])=K_1^G(B)$
implies that the map $K_1^G(B)\to K_1^G\bigl(B((x))\bigr)$ is surjective.
\end{proof}

\section{$\Aff^1$-homotopic interpretation of $K_1^G$}

Let $A$ be any commutative ring, $\Sm_A$ be the category of finitely presented smooth $A$-schemes,
and $\Sm_A^{\mathrm{aff}}$ be the full subcategory of affine schemes.

For any presheaf $F$ on $\Sm_A$, we denote by $\Sing_{\bullet}^{\A1}(F)$ the simplicial presheaf
$U\mapsto F(\Delta^*\times U)$, where $\Delta^*$ is the standard cosimplicial object made of affine
spaces (see e.g.~\cite[p. 88]{MoV}).

\begin{dfn}
Let $G$ be a group-valued presheaf on $\Sm_A$. For any $n\ge 1$, define \emph{the
$n$-th Karoubi--Villamayor $K$-theory functor associated to $G$}
to be the presheaf on $\Sm_A$ given by
$$
KV_n^G(U)=\pi_{n-1}\bigl(\Sing_{\bullet}^{\A1}(G)(U)\bigr).
$$
\end{dfn}

This definition goes back to J. F. Jardine~\cite{J} who
defined the non-stable Karoubi--Villamyor $K$-theory associated to a group valued
functor on a category of $k$-algebras, where $k$ is any commutative unitary ring,
similarly to Gersten's definition of the usual Karoubi--Villamayor $K$-theory.
For any commutative $A$-algebra $R$, one can explicitly compute $KV_1^G(R)$ as follows:
$$
KV_1^G(R)=G(R)/\{g\in G(R)\ |\ \exists\ h(x)\in G(R[x])\ :\ h(0)=1,\ h(1)=g\}.
$$
Let $G$ be a reductive group scheme over $A$, and let $P$ be a proper parabolic subgroup of $G$, and let $U_P$ be the
unipotent radical of $P$.
By~\cite[Exp. XXVI Corollaire 2.5]{SGA3} $U_P$ is $A$-schematically isomorphic to the canonical scheme of a projective
$A$-module $V$, with the point $0\in V$ corresponding to $1\in U_P(A)$. Consequently, every $g\in U_P(R)$ corresponds to an element $v\in V\otimes_A R$,
and then $vx\in V\otimes_A R[x]$ corresponds to $h(x)\in U_P(R[x])$ such that $h(0)=1$, $h(x)=g$.
This implies that there is a natural map of pointed sets
$$
K_1^{G,P}(R)=G(R)/E_P(R)\to KV_1^G(R).
$$
If $G$ has isotropic rank $\ge 2$, this map becomes a group homomorphism
$$
K_1^G(R)\to KV_1^G(R).
$$

A. Asok, M. Hoyois and M. Wendt proved the following ``affine representability'' result for $KV_1^G$ in the
Morel--Voevodsky $\Aff^1$-homotopy category $\Hcat_A$ over $A$, which is the homotopy category
of the category of simplicial sheaves in the Nisnevich topology on $\Sm_A$ with respect to the $\Aff^1$-model
structure~\cite[p. 109]{MoV}. Note that the definition of $\Hcat_A$ used in~\cite{AHW} and its prequel~\cite{AHW-1}
is different from the definition of~\cite{MoV}, however, the two categories are equivalent under the
assumption that $A$ is Noetherian and has finite Krull dimension~\cite[Remarks 3.1.4 and 5.1.1]{AHW-1}.

\begin{thm}\cite[Theorem 2.4.2]{AHW}\label{thm:AHW}
Let $A$ is a Noetherian scheme of finite Krull dimension.
Let $G$ be  finitely presented smooth $A$-group scheme such that
the natural map $H^1_{Nis}(X,G)\to H^1_{Nis}(\Aff^1_X,G)$ is bijective for all $X\in\Sm_A^{\mathrm{aff}}$.
Then for all $X\in\Sm_A^{\mathrm{aff}}$ and $n\ge 0$ the canonical map
$KV_1^G(X)\to\Hom_{\Hcat_A}(X,G)$ is bijective.
\end{thm}

If $A$ is a regular ring containing a field, then
the condition on Nisnevich cohomology in Theorem~\ref{thm:AHW} is satisfied for reductive $A$-groups
of isotropic rank $\ge 1$. This follows from
the Serre--Grothendieck conjecture~\cite{Pa-izvGS} and the extended version of the Serre--Grothendieck
conjecture for simply connected isotropic groups~\cite[Theorem 2.6]{Pa-GSfin1}. The next three statements
provide details of this implication.

\begin{lem}\label{lem:4maps}
Let $A$ be a regular semilocal domain, let $K$ be the fraction field of $A$, and let $G$ be a reductive group scheme over $A$.
Let $\der(G)$ be the derived group of $G$ in the sense of~\cite{SGA3}, and let
$G^{sc}$ be the simply connected semisimple group corresponding to $\der(G)$. Assume that
the three maps
\begin{equation}\label{eq:ma}
H^1_{\et}(A,G)\to H^1_{\et}(K,G),
\end{equation}
\begin{equation}\label{eq:mb}
H^1_{\et}(A,\der(G))\to H^1_{\et}(K,\der(G)),
\end{equation}
\begin{equation}\label{eq:mc}
H^1_{\et}(A[x_1,\ldots,x_n],G^{sc})\to H^1_{\et}(K[x_1,\ldots,x_n],G^{sc})
\end{equation}
 have trivial kernels. Then the map
\begin{equation}\label{eq:md}
H^1_{\et}(A[x_1,\ldots,x_n],G)\to H^1_{\et}(K[x_1,\ldots,x_n],G)
\end{equation}
has trivial kernel.
\end{lem}
\begin{proof}
There are two short exact sequences of reductive $A$-groups
$$
1\to\der(G)\to G\to\corad(G)\to 1
$$
and
$$
1\to C\to G^{sc}\to \der(G)\to 1,
$$
where $\corad(G)$ and $C$ are $A$-groups of multiplicative type~\cite[Exp. XXII]{SGA3}. By~\cite[Lemma 4.1]{St-dh}
the second sequence and the fact that~\eqref{eq:mb} and~\eqref{eq:mc} have trivial kernels imply that
\begin{equation}\label{eq:me}
H^1_{\et}(A[x_1,\ldots,x_n],\der(G))\to H^1_{\et}(K[x_1,\ldots,x_n],\der(G))
\end{equation}
has trivial kernel. Also by~\cite[Lemma 4.1]{St-dh} the first sequence and the fact that~\eqref{eq:ma} and~\eqref{eq:me}
have trivial kernels imply that~\eqref{eq:md} has trivial kernel.
\end{proof}

\begin{lem}\label{lem:nis}
Let $A$ be a regular semilocal domain containing a field $k$, let $K$ be the fraction field of $A$, and let $G$ be
a reductive group scheme over $A$ of isotropic rank $\ge 1$. Then the natural map $H^1_{\et}(A[x],G)\to H^1_{\et}(K[x],G)$
has trivial kernel.
\end{lem}
\begin{proof}
By Lemma~\ref{lem:4maps} it is enough to check that~\eqref{eq:ma},~\eqref{eq:mb},~\eqref{eq:mc} have trivial kernels.
By the Serre--Grothendieck conjecture over $A$~\cite[Theorem 1.1]{Pa-izvGS}
the maps~\eqref{eq:ma} and~\eqref{eq:mb} have trivial kernels. Since $G^{sc}$ has isotropic rank $\ge 1$,
by~\cite[Theorem 2.6]{Pa-GSfin1} the map~\eqref{eq:mc} has trivial kernel,
under the additional assumption that $A$ is a semilocal ring of several closed points on
a $k$-smooth irreducible affine variety $X$. Using Popescu's theorem as in the proof of Theorem~\ref{thm:A1-inv},
and the fact that $H^1_{\et}(-,G)$ commutes with filtered direct limits~\cite{Margaux-Chech}, we conclude
that~\eqref{eq:mc} has trivial kernel for any regular semilocal domain $A$ containing a field.
\end{proof}

\begin{cor}\label{cor:nis}
Let $A$ be a regular ring containing a field $k$, and let $G$ be
a reductive group scheme over $A$ of isotropic rank $\ge 1$. Then the natural map $H^1_{Nis}(A,G)\to H^1_{Nis}(A[x],G)$
is bijective.
\end{cor}
\begin{proof}
It is enough to show that any Nisnevich $G$-torsor over $A[x]$ is extended from $A$.
By~\cite[Corollary 3.2]{Thomason} $G$
is linear, hence by the local-global principle for torsors~\cite[Theorem 3.2.5]{AHW} (see also~\cite[Korollar 3.5.2]{Mo})
it is enough to prove that every $G$-torsor over $A_m[x]$ is extended from $A_m$,
for every maximal localization $A_m$ of $A$. Thus, we can assume from the start that $A$ is regular local.

Next, for every regular local ring $A$ containing a field, we show that every Nisnevich $G$-torsor over $A[x]$ is
in fact trivial, and hence extended. By Lemma~\ref{lem:nis}
$H^1_{\et}(A[x],G)\to H^1_{\et}(K[x],G)$ has trivial kernel, where $K$ is the fraction field
of $A$.
By~\cite[Proposition 2.2]{CTO} $H^1_{\et}(K[x],G)\to H^1_{\et}(K(x),G)$ has trivial kernel.
Since $H^1_{Nis}(K(x),G)=1$, every Nisnevich $G$-torsor over $A[x]$ is trivial.
\end{proof}

\begin{thm}\label{thm:repr}
Let $A$ be a regular ring of finite Krull dimension and containing a field $k$.
Let $G$ be a reductive group scheme over $A$ of isotropic rank $\ge 2$. Then the canonical
map $K_1^G(A)\to \Hom_{\Hcat_A}(\Spec(A),G)$ is bijective.
\end{thm}
\begin{proof}
The map $K_1^G(A)\to KV_1^G(A)$ is
bijective by~\cite[Lemma 3.3]{St-poly}, given that $K_1^G(A[x])\cong K_1^G(A)$ by Theorem~\ref{thm:A1-inv}.
The map $KV_1^G(A)\to\Hom_{\Hcat_A}(\Spec(A),G)$
is bijective by Theorem~\ref{thm:AHW} and Corollary~\ref{cor:nis}.
\end{proof}

\renewcommand{\refname}{References}

\end{document}